\documentclass[12pt]{article}

\usepackage[T1]{fontenc}
\usepackage[utf8]{inputenc}
\usepackage{lmodern}
\usepackage{amsmath,amssymb,amsthm,mathtools}
\usepackage{booktabs}
\usepackage{enumitem}
\usepackage{geometry}
\usepackage{microtype}
\usepackage{xcolor}
\usepackage{xurl}
\usepackage[colorlinks=true,linkcolor=blue!45!black,citecolor=blue!45!black,urlcolor=blue!45!black]{hyperref}
\usepackage[nameinlink,noabbrev]{cleveref}

\setlist{itemsep=2pt,topsep=4pt}

\newtheorem{theorem}{Theorem}[section]
\newtheorem{proposition}[theorem]{Proposition}
\newtheorem{lemma}[theorem]{Lemma}
\newtheorem{corollary}[theorem]{Corollary}

\theoremstyle{definition}
\newtheorem{definition}[theorem]{Definition}
\newtheorem{example}[theorem]{Example}

\theoremstyle{remark}

\newcommand{\wt}{\operatorname{wt}}
\newcommand{\dist}{\operatorname{dist}}

\title{Pair-Defensive Silver Colourings of Hypercubes}
\author{Saman Farhat%
    \thanks{The Graduate Center of the City University of New York. Email:
        \texttt{saman.farhat33@gc.cuny.edu}.}%
    \thanks{Amirkabir University of Technology. Email:
        \texttt{saman.farhat@aut.ac.ir}.}
    \and
    Mehrak Shirkhani%
    \thanks{University of Tehran. Email:
        \texttt{mehrakshirkhani@ut.ac.ir}.}}
\hypersetup{pdftitle={Pair-Defensive Silver Colourings of Hypercubes},
    pdfauthor={Saman Farhat, Mehrak Shirkhani}}
\date{}

\begin{document}
\maketitle

\begin{abstract}
    A silver colouring is a proper colouring in which every colour appears in the
    closed neighbourhood of each vertex of a prescribed independent set. This
    local condition suffices when vertices are tested one at a time. We study a
    stronger requirement for simultaneous testing: whenever one or two vertices of
    the independent set are attacked together, each colour must supply distinct
    nearby defenders for them. We call this a \emph{pair-defensive silver
    colouring}. For the hypercube \(Q_d\), with one parity class as the attacked
    set, we prove that the maximum number of colours in a pair-defensive silver
    colouring is at most
    $\left\lfloor\frac{d+3}{2}\right\rfloor$,
    roughly half the ordinary silver-colouring target \(d+1\). We construct
    colourings attaining this bound in four consecutive dimensions around every
    power of two, and we study the structure of the extremal, bound-attaining
    colourings. In each odd critical dimension, we characterize the extremal
    colourings by a partition of the defender parity into regular,
    triangle-free subgraphs of the halved cube. This characterization also has
    a local form in terms of coordinate matchings and a defect coordinate.
\end{abstract}

\section{Introduction}

A silver colouring asks for local colour coverage. In an \(r\)-regular graph,
one wants a proper colouring with \(r+1\) colours such that every colour
appears in the closed neighbourhood of each vertex in a prescribed maximum
independent set. This condition is natural when vertices are tested one at a
time: around every tested vertex, all colours are visible.

This paper studies the situation in which up to two vertices may be attacked
simultaneously. For each possible attack, every colour must have distinct
representatives in the closed neighbourhoods of the attacked vertices; these
representatives serve as defenders of the corresponding types. Local colour
coverage alone is not sufficient for this purpose. The same nearby vertex may
be the only representative of its colour for both attacked vertices, and thus
cannot defend them both at once. Pair defense therefore strengthens the
silver-colouring condition by requiring distinct defenders for simultaneous
attacks.

More generally, an \emph{attack} is a set of at most \(k\) vertices that may
be attacked simultaneously. A defensive silver colouring is a proper
colouring together with a fixed independent set \(I\) of possible attacked
vertices. For every attack \(S\subseteq I\) with \(|S|\le k\), each colour
class must provide distinct defenders for the vertices of \(S\); a defender
may be the attacked vertex itself or one of its neighbours. In other words,
the available defenders of each colour must satisfy Hall's matching condition
for every allowed attack. Most of this paper concerns the first nontrivial
case, \(k=2\), which we call \emph{pair defense}. The formal definition
appears in \cref{def:defensive-silver}.

Our main setting is the hypercube \(Q_d\). Its vertices are the binary words
of length \(d\), with two words adjacent precisely when they differ in one
coordinate. The vertices split into two parity classes according to the parity
of their Hamming weight, and each parity class is a maximum independent set.
We fix one parity class, denoted by \(I\), as the set of possible attacked
vertices. We write \(\sigma_k(Q_d)\) for the largest number of colours in a
proper colouring of \(Q_d\) that is a defensive silver colouring for every
set \(S\subseteq I\) with \(|S|\le k\). The opposite parity class, denoted \(D\), supplies
most defenders, although an attacked vertex may also defend itself.

The classical starting point is the silver-matrix problem, posed as
Problem~4 of the 1997 International Mathematical Olympiad. Mahdian and
Mahmoodian discussed its roots and its connection to defining sets in graph
colourings in \emph{The Roots of an IMO97 Problem} \cite{mahdian2000roots}.
Ghebleh, Goddyn,
Mahmoodian, and Verdian-Rizi developed the graph-theoretic silver-cube
formulation in \emph{Silver Cubes} \cite{ghebleh2008silver}. They studied
silver colourings of Cartesian products of complete graphs, including
hypercubes, using Hamming-code constructions and coding-theoretic bounds. 
Further work includes silver block-intersection graphs of Steiner designs
\cite{ahadi2013silver} and totally silver graphs, where every closed
neighbourhood contains each colour exactly once
\cite{ghebleh2013totally}. None of these works imposes the distinct-defender requirement for simultaneous
attacks that defines defensive silver colourings.

The second source of motivation is defensive domination. Farley and
Proskurowski introduced \(k\)-defensive domination
\cite{farley2004defensive}, where a set of defenders must assign distinct
defenders to every attack of at most \(k\) vertices. Defensive domination on
interval graphs was also studied by
Dereniowski, Gaven{\v c}iak, and Kratochv{\'i}l
\cite{dereniowski2019cops}. Complexity and algorithmic results were obtained
by Ekim, Farley, Proskurowski and coauthors
\cite{ekim2020complexity,ekim2023proper}, and more recently by Varol, Ekim,
and Tan{\i}nm{\i}{\c s} \cite{varol2026benders}. Pair-defensive silver
colourings combine these two ideas by requiring every colour class of a
proper colouring to serve independently as a defensive set.

Our first main result gives a general upper bound on the number of colours
that can be used while preserving pair defense:
\[
\sigma_2(Q_d)\le
\left\lfloor\frac{d+3}{2}\right\rfloor .
\]
This is much smaller than the ordinary silver target \(d+1\). The reduction
is the price of robustness: to defend simultaneous attacks, most colours must
be represented more than once near an attacked vertex. For any \(r\)-regular
graph \(G\) with a maximum independent set \(I\), an ordinary silver colouring
may use \(r+1\) colours; the hypercubes \(Q_{2^t}\) for \(t\ge1\) and
\(Q_{2^t-1}\) for \(t\ge2\) provide examples
\cite[Corollary~5.1]{ghebleh2008silver}. Under
pair defense at the same number of colours, this collapses to
the trivial case of a disjoint union of complete graphs: every connected
component of \(G\) must be \(K_{r+1}\) (\cref{thm:collapse}). Thus pair defense
is genuinely more restrictive even beyond hypercubes.

We call a colouring \emph{extremal} when it attains the pair bound. As in the
classical silver-colouring problem, equality cases reveal the underlying
structure and are central to our study. For every power of two \(p\ge2\), we
construct extremal colourings in the four consecutive dimensions
(\cref{thm:dyadic})
\[
2p-3,\quad 2p-2,\quad 2p-1,\quad 2p.
\]

We focus especially on the odd critical dimension \(d=2q-3\). There, an
extremal \(q\)-colouring has a rigid neighbourhood profile: at every attacked
vertex, the open neighbourhood has colour multiplicities
\[
(0,1,2,\ldots,2),
\]
up to permutation. This profile yields a partition theorem. If \(B_i\) is the
set of defender-parity vertices of colour \(i\), then
\(B_1,\ldots,B_q\) form an ordered near-partition of the defender parity
\(D\). Each nonempty block induces a triangle-free \((q-2)\)-regular
subgraph of the \emph{halved cube} on \(D\), in which two vertices are
adjacent when their Hamming distance is two. Conversely, every such
partition reconstructs the full pair-defensive colouring uniquely
(\cref{thm:odd-partition}). Thus we obtain a complete characterization of
extremal pair-defensive \(q\)-colourings in the \(d=2q-3\) dimension.

The same equality structure has a local coordinate interpretation in odd
dimensions. At each defender vertex, the coordinates are naturally paired up,
except for one uncovered coordinate, called the defect coordinate. These
defect edges form a perfect matching between the two parts of the hypercube.
The local matching condition also characterizes extremal colourings directly
from the defender-side colouring, and a constant defect coordinate gives a
reduction to a colouring of a smaller binary cube
(\cref{cor:local-normal-form,prop:constant-defect}).

The paper is organized as follows. Section~\ref{sec:defensive-silver}
introduces the definitions and Hall formulation. Section~\ref{sec:bound}
proves the pair bound and gives the main constructions.
Sections~\ref{sec:odd-equality} and \ref{sec:local-geometry} develop the equality
theory and local matching geometry. The final section summarizes these
results.

\section{Defensive silver colourings and Hall's formulation}
\label{sec:defensive-silver}

Throughout, let \(G=(V(G),E(G))\) be a finite, nonempty, simple, undirected
graph. For a vertex \(x\in V(G)\), its open neighbourhood is
\[
N(x)=\{y\in V(G):xy\in E(G)\},
\]
and its closed neighbourhood is
\[
N[x]=N(x)\cup\{x\}.
\]
For a set \(X\subseteq V(G)\), define
\[
N[X]=\bigcup_{x\in X}N[x].
\]
The square \(G^2\) is the graph on \(V(G)\) in which two distinct vertices
are adjacent whenever their distance in \(G\) is at most two. We write
\[
[q]=\{1,\ldots,q\}.
\]
If \(c:V(G)\to[q]\) is a colouring, its colour-\(i\) class is
\[
C_i=c^{-1}(i).
\]
A colouring \(c:V(G)\to[q]\) is \emph{proper} if adjacent vertices receive
different colours. A set \(I\subseteq V(G)\) is \emph{independent} if no two
vertices of \(I\) are adjacent.
\begin{definition}\label{def:defensive-silver}
    Let $I\subseteq V(G)$ be an independent set and let $k\ge0$. An
    \emph{attack} is a set $S\subseteq I$ with $|S|\le k$. We say that a
    surjective proper colouring
    $c:V(G)\to[q]$ is an \emph{$(I,k,q)$-defensive silver colouring} if, for
    every colour $i\in[q]$ and every attack $S\subseteq I$, there exists an
    injective map
    \[
    f_{i,S}:S\longrightarrow C_i
    \]
    such that
    \[
    f_{i,S}(x)\in N[x]
    \qquad\text{for every }x\in S.
    \]
    Thus, for each attack, every colour class provides a distinct defender for
    each attacked vertex. The empty attack is allowed and is defended by the
    empty map.
    
    We write $\sigma_k(G,I)$ for the largest number of colours in an
    $(I,k,q)$-defensive silver colouring of $G$.
\end{definition}

If $k>|I|$, the definition does not change when $k$ is replaced by $|I|$,
since every attack is a subset of $I$. Self-defense is allowed, which is why
we use closed neighbourhoods.

We identify the \(d\)-dimensional hypercube with the graph \(Q_d\) on
\(\mathbb F_2^d\), where two vertices are adjacent when they differ in
exactly one coordinate. Here \(\mathbb F_2=\{0,1\}\), and vector addition is
performed modulo \(2\). The graph is bipartite: its two parity classes are
\[
I=\{x\in\mathbb F_2^d:\wt(x)\equiv0\pmod2\},
\qquad
D=\{x\in\mathbb F_2^d:\wt(x)\equiv1\pmod2\},
\]
where \(\wt(x)\) denotes the number of nonzero coordinates of \(x\).
We regard \(I\) as the attacked parity and \(D\) as the defender parity.
Choosing the other parity class as \(I\) gives an equivalent problem, since
the two choices are related by a cube automorphism. We write
\[
\sigma_k(Q_d)=\sigma_k(Q_d,I).
\]
For \(j\in[d]\), let \(e_j\) denote the \(j\)th standard basis vector.
The Hamming distance \(\dist(x,y)\) is the number of coordinates in which
\(x\) and \(y\) differ. 

The following lemma gives an equivalent formulation of the definition in
terms of Hall's theorem \cite[Observation~2]{ekim2023proper}.

\begin{lemma}[Hall formulation]
    \label{lem:hall}
    A colour class $C_i$ defends every attack $S\subseteq I$ with $|S|\le k$ if
    and only if
    \[
    |N[S]\cap C_i|\ge |S|
    \]
    for every such set $S$.
\end{lemma}

\begin{proof}
    Fix a colour $i$ and an attack $S$. Consider the bipartite graph with two
    sides $S$ and $C_i$, where a vertex $x\in S$ is adjacent to a
    vertex $z\in C_i$ exactly when $z\in N[x]$. A defense of $S$ by colour $i$
    is precisely a matching in this bipartite graph that saturates $S$.
    
    By Hall's theorem, such a matching exists if and only if
    \[
    |N[T]\cap C_i|\ge |T|
    \]
    for every subset $T\subseteq S$. Therefore colour $i$ defends every attack
    of size at most $k$ if and only if this inequality holds for every
    $S\subseteq I$ with $|S|\le k$.
\end{proof}

For $k=2$, we call the inequalities for attacks of size one and two
\emph{singleton Hall} and \emph{pair Hall}, respectively.
The next lemma reduces the number of attacks that must be checked
\cite[Observation~1]{ekim2023proper}. Recall that
\(G^2\) has the same vertex set as \(G\), with two vertices adjacent whenever
their distance in \(G\) is at most two. For \(S\subseteq V(G)\), we write
\(G^2[S]\) for the subgraph of \(G^2\) induced by the vertices in \(S\).
Thus, \(G^2[S]\) is connected when the vertices of \(S\) can be joined by a
chain whose consecutive vertices are at distance at most two in \(G\). This
is different from requiring \(S\) to be connected in \(G\). If an attack
splits into several \(G^2\)-connected parts, then those parts are sufficiently
far apart that their Hall inequalities combine automatically.

\begin{lemma}[Connected-square reduction]
	\label{lem:connected-square}
	To verify the Hall condition in \cref{lem:hall}, it is enough to check
	nonempty attacks \(S\subseteq I\) for which the induced subgraph
	\(G^2[S]\) is connected.
\end{lemma}

\begin{proof}
    Let $S_1,\ldots,S_t$ be the vertex sets of the connected components of
    $G^2[S]$. Suppose that a vertex $z$ lies in both $N[S_a]$ and $N[S_b]$ for
    some $a\ne b$. Then there are vertices $x\in S_a$ and $y\in S_b$ such that
    $z\in N[x]\cap N[y]$. Hence
    \[
    \dist_G(x,y)\le 2,
    \]
    so $x$ and $y$ are adjacent in $G^2$. This contradicts the fact that
    $S_a$ and $S_b$ are distinct components of $G^2[S]$.
    
    Therefore the sets $N[S_1],\ldots,N[S_t]$ are pairwise disjoint. If the Hall
    inequality holds for each component, then
    \[
    |N[S]\cap C_i|
    =
    \sum_{j=1}^t |N[S_j]\cap C_i|
    \ge
    \sum_{j=1}^t |S_j|
    =
    |S|.
    \]
    Thus the Hall inequality also holds for $S$.
\end{proof}

For hypercubes, this disjointness will often allow us to work component by
component within a single parity class. We therefore define the \emph{halved
	cube} on \(D\) to be the graph \(H_d\) with vertex set \(D\), in which two
vertices are adjacent precisely when their Hamming distance in \(Q_d\) is
two. The same construction can be made on \(I\), and we use whichever parity
class is relevant.

We next record how the definition specializes to the classical notion of a
silver colouring when only one vertex is attacked.

\begin{proposition}[The silver endpoint]
    \label{prop:silver-endpoint}
    Let $G$ be an $r$-regular graph, and let $I$ be a maximum independent set.
    An $(I,1,r+1)$-defensive silver colouring is exactly a silver colouring with
    respect to $I$ in the sense of \cite{ghebleh2008silver}.
\end{proposition}

\begin{proof}
    For $k=1$, the only nonempty attacks are single vertices. Thus singleton
    defense says that, for every $x\in I$, each colour appears at least once in
    $N[x]$. Since $G$ is $r$-regular, the closed neighbourhood $N[x]$ has size
    $r+1$. There are also exactly $r+1$ colours, so each colour appears exactly
    once in $N[x]$. This is precisely the silver-colouring condition with
    respect to $I$. The converse is immediate.
\end{proof}

Pair defense at the same number of colours is much more restrictive.

\begin{theorem}[Collapse at $r+1$ colours]
    \label{thm:collapse}
    Let $G$ be an $r$-regular graph, and let $I$ be a maximum independent set.
    If $G$ admits an $(I,2,r+1)$-defensive silver colouring, then every connected
    component of $G$ is a copy of $K_{r+1}$.
\end{theorem}

\begin{proof}
    We compare two bounds on $|V(G)|$. Fix a colour $i$. By singleton defense,
    each vertex $x\in I$ has at least one colour-$i$ vertex in $N[x]$. Since
    $G$ is $r$-regular and there are $r+1$ colours, every colour appears exactly
    once in each closed neighbourhood $N[x]$. Let $f_i(x)$ be the unique
    colour-$i$ vertex in $N[x]$.
    
    We claim that $f_i$ is injective. Indeed, if $f_i(x)=f_i(y)$ for distinct
    vertices $x,y\in I$, then the attack $\{x,y\}$ would have only one available
    defender of colour $i$, contradicting pair Hall. Hence $|C_i|\ge |I|$ for
    every colour $i$, and therefore
    \[
    |V(G)|=\sum_{i=1}^{r+1}|C_i|\ge (r+1)|I|.
    \]
    
    On the other hand, since $I$ is maximum, it is maximal. Thus every vertex in
    $V(G)\setminus I$ has at least one neighbour in $I$. Counting the edges from
    $I$ to $V(G)\setminus I$ gives
    \[
    |V(G)|-|I|\le r|I|,
    \]
    or equivalently
    \[
    |V(G)|\le (r+1)|I|.
    \]
    Together, the two bounds force equality throughout.
    
    In particular, every vertex outside $I$ has exactly one neighbour in $I$.
    Thus the closed neighbourhoods $N[x]$, for $x\in I$, form a partition of
    $V(G)$.
    
    It remains to identify each part. Fix $x\in I$. If two vertices in $N(x)$
    were nonadjacent, then replacing $x$ in $I$ by those two vertices would give
    a larger independent set.
    
    Finally, every vertex in this clique already has its full degree $r$ inside
    the clique. Therefore there are no edges between different parts of the
    partition. Hence every connected component of $G$ is a copy of $K_{r+1}$.
\end{proof}

The maximum-independent-set hypothesis in \cref{thm:collapse} is essential.
If $I$ is a singleton, then there are no two-vertex attacks to check. There
are also counterexamples with two-vertex attacks: take the cycle $C_6$, let
$I$ consist of two antipodal vertices, and colour the vertices cyclically by
\[
0,1,0,2,1,2.
\]
This gives an $(I,2,3)$-defensive silver colouring, but $C_6$ is not a
disjoint union of triangles. The same idea extends to every $C_n$ for
$n\ge6$: subdivide an edge joining the two disjoint closed neighbourhoods,
colouring each inserted vertex properly with one of the three colours. The
attacked vertices retain their disjoint rainbow closed neighbourhoods.

\section{The pair-defense bound and constructions}
\label{sec:bound}

Fix a parity class $I$ of $Q_d$, and put
\[
D=V(Q_d)\setminus I,
\qquad
m=|I|=|D|=2^{d-1}.
\]
We call $D$ the \emph{defender parity}. Vertices of $I$ may still defend
themselves, since closed neighbourhoods are used.

We begin with two elementary bounds. The second one is useful when large
attacks are allowed.

\begin{proposition}
    \label{prop:elementary-bounds}
    For $d\ge1$ and $k\ge1$, let
    \[
    t=\min\{k,2^{d-1}\}.
    \]
    Then
    \[
    2\le \sigma_k(Q_d)\le
    \min\left\{
    d+1,
    \left\lfloor\frac{2^d}{t}\right\rfloor
    \right\}.
    \]
\end{proposition}

\begin{proof}
    Colour the two parity classes of $Q_d$ with two different colours. This gives
    a proper colouring. The colour of $I$ defends attacks by self-defense. The
    colour of $D$ defends attacks by sending each attacked vertex $x\in I$ a defender from $x+e_1\in D$. Thus $\sigma_k(Q_d)\ge2$.
    
    For the first upper bound, singleton Hall implies that every colour must
    appear in the closed neighbourhood of each vertex of $I$. Each such closed
    neighbourhood has size $d+1$, so there can be at most $d+1$ colours.
    
    For the second upper bound, choose an attack $S\subseteq I$ of size
    $t=\min\{k,|I|\}$. By Hall's condition, each colour class $C_i$ satisfies
    \[
    |N[S]\cap C_i|\ge t.
    \]
    In particular, $|C_i|\ge t$ for every colour $i$. If there are $q$ colours,
    then
    \[
    2^d=|V(Q_d)|=\sum_{i=1}^q |C_i|\ge qt,
    \]
    so
    \[
    q\le \left\lfloor\frac{2^d}{t}\right\rfloor .
    \]
\end{proof}

We now prove the main counting bound behind the pair-defense restriction.
The idea is to count, for each attacked vertex, how many times each colour
appears in its closed neighbourhood. As we will see, pair Hall limits how often a colour can
appear exactly once. Notice that $r+1$ is an upper bound that follows from singleton defense. 
We prove another upper bound below.

\begin{theorem}[Regular-graph pair bound]
    \label{prop:regular-pair-bound}
    Let $G$ be an $r$-regular graph, and let $I$ be a nonempty independent set.
    Every $(I,2,q)$-defensive silver colouring satisfies
    \[
    q\le
    \left\lfloor
    \frac{r+1+|V(G)|/|I|}{2}
    \right\rfloor .
    \]
    In particular, if $G$ is bipartite with equal parts and $I$ is one whole
    part, then
    \[
    q\le \left\lfloor\frac{r+3}{2}\right\rfloor .
    \]
\end{theorem}

\begin{proof}
    For $x\in I$ and $i\in[q]$, define
    \[
    a_i(x)=|N[x]\cap C_i|.
    \]
    By singleton Hall, we have $a_i(x)\ge1$ for every $x\in I$ and every
    colour $i$.
    
    Call the pair $(x,i)$ a \emph{singleton incidence} if $a_i(x)=1$. Let $s_i$
    be the number of singleton incidences of colour $i$, and put
    \[
    S=\sum_{i=1}^q s_i.
    \]
    For each singleton incidence $(x,i)$, there is a unique vertex of colour
    $i$ in $N[x]$; call it the unique defender of $x$ in colour $i$.
    
    We claim that, for each fixed colour $i$, these unique defenders are all
    distinct. Indeed, if two different vertices $x,y\in I$ had the same unique
    colour-$i$ defender, then
    \[
    |N[\{x,y\}]\cap C_i|=1,
    \]
    contradicting pair Hall. Therefore
    \[
    s_i\le |C_i|
    \]
    for each colour $i$, and hence
    \[
    S=\sum_{i=1}^q s_i\le \sum_{i=1}^q |C_i|=|V(G)|.
    \]
    
    Now count incidences between closed neighbourhoods of vertices in $I$ and
    colour classes. Since $G$ is $r$-regular,
    \[
    \sum_{x\in I}\sum_{i=1}^q a_i(x)
    =
    \sum_{x\in I}|N[x]|
    =
    |I|(r+1).
    \]
    For a fixed $x$, every colour contributes at least one. Singleton incidences
    contribute exactly one, while all other colours contribute at least two.
    Thus
    \[
    \sum_{x\in I}\sum_{i=1}^q a_i(x)
    \ge
    2q|I|-S.
    \]
    Using $S\le |V(G)|$, we obtain
    \[
    |I|(r+1)
    \ge
    2q|I|-|V(G)|.
    \]
    Rearranging gives
    \[
    q\le
    \frac{r+1+|V(G)|/|I|}{2}.
    \]
    Since $q$ is an integer, the stated bound follows.
    
    If $G$ is bipartite with equal parts and $I$ is one whole part, then
    $|V(G)|/|I|=2$, giving
    \[
    q\le \left\lfloor\frac{r+3}{2}\right\rfloor .
    \]
\end{proof}

\begin{corollary}[Pair bound]
\label{thm:pair-bound}
For every $d\ge1$ and every $k\ge2$,
\[
\sigma_k(Q_d)\le
\left\lfloor\frac{d+3}{2}\right\rfloor .
\]
\end{corollary}

\begin{proof}
For $k\ge2$, every $(I,k,q)$-defensive silver colouring is, in particular,
pair-defensive. We may therefore apply \cref{prop:regular-pair-bound} to
$Q_d$, with $r=d$ and
\[
|V(Q_d)|/|I|=2.
\]
This gives
\[
q\le \left\lfloor\frac{d+3}{2}\right\rfloor .
\]
\end{proof}

We shall need the equality conditions in the proof of
\cref{prop:regular-pair-bound}. Retain the notation
\[
a_i(x)=|N[x]\cap C_i|,
\]
let $s_i$ be the number of singleton incidences of colour $i$, and let
\[
S=\sum_{i=1}^q s_i.
\]
For a colouring of $Q_d$, with $m=|I|=|D|$, the counting argument gives
\[
m(d+1)
=
\sum_{x\in I}\sum_{i=1}^q a_i(x)
\ge
S+2(qm-S)
=
2qm-S
\ge
2qm-2m.
\]
When \(d=2q-3\) and a \(q\)-colouring attains the pair bound, we have
\(m(d+1)=2qm-2m\), so both inequalities in the counting argument are
equalities. In this odd critical dimension, every colour that appears more
than once in a closed neighbourhood appears exactly twice. Moreover, the
singleton incidences account for all \(2|I|\) vertices: the injection that
assigns each singleton incidence to its unique defender is therefore a
bijection. Properness already implies that each \(x\in I\) is the unique
defender of its own colour in \(N[x]\). Consequently, every vertex of the
opposite parity class \(D\) is the unique defender of some attacked vertex in
\(I\). In the even critical dimension \(d=2q-2\), attaining the rounded bound
does not force equality in these inequalities.

Below we record two construction mechanisms. The first is a linear construction on
the whole cube. Here the colours are the elements of $\mathbb F_2^h$.

\begin{theorem}[Binary-linear criterion]
	\label{thm:linear-criterion}
	Let \(h\ge1\), put \(q=2^h\), and assign a label
	\[
	\lambda_j\in\mathbb F_2^h
	\]
	to each coordinate \(j\in[d]\). Define
	\[
	c(x)=\sum_{j=1}^d x_j\lambda_j,
	\qquad x=(x_1,\ldots,x_d)\in Q_d.
	\]
	Then \(c\) is proper, surjective, and pair-defensive with respect to either
	parity class if and only if the following conditions hold:
	
	\begin{enumerate}[label=\textup{(\roman*)}]
		\item no label \(\lambda_j\) is zero;
		\item every nonzero vector of \(\mathbb F_2^h\) occurs among the labels;
		\item at most one nonzero vector occurs exactly once.
	\end{enumerate}
	
	Consequently, within this linear construction, the minimum possible
	dimension is
	\[
	d= 1+ 2(q-2) = 2q - 3.
	\]
\end{theorem}

\begin{proof}
	We first verify singleton defense and then pair defense.
	
	Translate an attacked vertex to \(0\). This translation may exchange the two
	parity classes, which is harmless because the argument is local. Adjacent
	vertices differ by some \(e_j\), and
	\[
	c(x+e_j)-c(x)=\lambda_j.
	\]
	Thus the colouring is proper if and only if no label \(\lambda_j\) is zero,
	which is condition~\textup{(i)}.
	
	At the vertex \(0\), the colours on the open neighbourhood are precisely
	\(\lambda_1,\ldots,\lambda_d\). Hence singleton defense holds if and only if
	every nonzero vector of \(\mathbb F_2^h\) occurs among the labels, giving
	condition~\textup{(ii)}. Since \(c(0)=0\), this condition also implies
	surjectivity.
	
	It remains to analyse pair defense. Two vertices in the same parity class
	have intersecting closed neighbourhoods only when their distance is two.
	After translation, such a pair can be written as
	\[
	\{0,e_r+e_s\}
	\]
	for distinct coordinates \(r\) and \(s\).
	
	Suppose that \(\lambda_r\neq\lambda_s\) and that both labels occur exactly
	once among \(\lambda_1,\ldots,\lambda_d\). We claim that \(e_r\) is the only
	defender of colour \(\lambda_r\) available to the attack
	\(\{0,e_r+e_s\}\).
	
	Indeed, the colours on the open neighbourhood of \(0\) are
	\(\lambda_1,\ldots,\lambda_d\). Since \(\lambda_r\) occurs only once, the
	only vertex of colour \(\lambda_r\) in \(N[0]\) is \(e_r\). Moreover,
	\(e_r\) is adjacent to \(e_r+e_s\), so it is also available to the second
	attacked vertex.
	
	It remains to check that no other vertex in \(N[e_r+e_s]\) has colour
	\(\lambda_r\). A neighbour of \(e_r+e_s\) has the form
	\(e_r+e_s+e_i\) and has colour
	\[
	\lambda_r+\lambda_s+\lambda_i.
	\]
	If this colour were \(\lambda_r\), then
	\[
	\lambda_i=\lambda_s.
	\]
	Because \(\lambda_s\) occurs exactly once, this forces \(i=s\), and the
	corresponding vertex is again \(e_r\). The two attacked vertices themselves
	do not have colour \(\lambda_r\), since all labels are nonzero and
	\(\lambda_r\neq\lambda_s\). Therefore
	\[
	\bigl|N[\{0,e_r+e_s\}]\cap C_{\lambda_r}\bigr|=1,
	\]
	which violates pair Hall. Thus at most one nonzero label can occur exactly once,
	which proves the necessity of condition~\textup{(iii)}.
	
	Conversely, suppose that pair Hall fails for some two-vertex attack. Each
	attacked vertex has at least one defender of the relevant colour by
	singleton defense. Hence a failure can occur only if the two attacked
	vertices have the same unique defender of some colour. Their closed
	neighbourhoods must therefore intersect, so the vertices are at distance
	two. Translating the pair to \(\{0,e_r+e_s\}\), the common unique defender
	can occur only when the two distinct labels \(\lambda_r\) and \(\lambda_s\)
	both occur exactly once, with an argument similar to the previous paragraph. This contradicts condition~\textup{(iii)}.
	Pairs at greater distance have disjoint closed neighbourhoods and therefore
	cannot cause a pair-Hall failure. Hence the colouring is pair-defensive.
	
	Finally, \(\mathbb F_2^h\) has \(q-1\) nonzero vectors. Condition~\textup{(ii)}
	requires all of them to occur, while condition~\textup{(iii)} allows at most
	one to occur once. Consequently,
	\[
	d\ge 1+2(q-2)=2q-3.
	\]
	Equality is attained by using one nonzero label once and every other nonzero
	label twice.
\end{proof}

\begin{example}[A binary-linear pair-defensive colouring of \(Q_5\)]
	\label{ex:q5-linear}
	Consider \(Q_5\) with four colours. Assign the following labels in
	\(\mathbb F_2^2\) to its five coordinates:
	\[
	(1,0),\ (1,0),\ (0,1),\ (1,1),\ (1,1).
	\]
	The resulting colouring is
	\[
	c(x_1,x_2,x_3,x_4,x_5)
	=
	(x_1+x_2+x_4+x_5,\,
	x_3+x_4+x_5),
	\]
	where all sums are taken modulo \(2\).
	
	One nonzero label occurs once, while each of the other two nonzero labels
	occurs twice. Therefore, by \cref{thm:linear-criterion}, this is a proper,
	surjective, pair-defensive \(4\)-colouring of \(Q_5\), with respect to
	either parity class. Since
	\[
	\left\lfloor\frac{5+3}{2}\right\rfloor=4,
	\]
	the pair-defense bound is attained, and hence
	\[
	\sigma_2(Q_5)=4.
	\]
\end{example}

The next construction is linear only on one parity class and is not an
instance of \cref{thm:linear-criterion}.

\begin{example}[$Q_7$ with five colours]
    \label{ex:q7}
    Label the seven coordinates by the multiset
    \[
    0,1,1,2,2,3,3
    \]
    in $\mathbb F_2^2$. Colour every even vertex with a new colour $\infty$.
    For an odd vertex $x$, define
    \[
    c(x)=\sum_j x_j\lambda_j
    \]
    and colour $x$ by $c(x)$.
    
    Consider an even attacked vertex $x$. Its seven neighbours have colours
    \[
    c(x+e_j)=c(x)+\lambda_j,
    \]
    so among these neighbours one finite colour appears once and each of the
    other three finite colours appears twice. The colour $\infty$ is represented
    by $x$ itself.
    
    The only possible singleton finite defender is obtained by flipping the
    coordinate whose label is zero. For each attacked vertex this finite
    colour has that unique defender, and the zero coordinate is fixed, so
    distinct attacked vertices have distinct singleton defenders. Every other
    finite colour has two defenders already in each open neighbourhood; the
    colour $\infty$ is defended by each attacked vertex itself. Thus singleton
    Hall and pair Hall hold, and the upper bound gives
    \[
    \sigma_2(Q_7)=5.
    \]
\end{example}

We next describe a dimension-lift construction. It allows a defensive silver
colouring of \(Q_d\) to be extended to one of \(Q_{d+1}\) without decreasing
the number of colours or the maximum allowed attack size. We use \(\square\) to denote the Cartesian product of graphs. The vertices of
\(G\square H\) are pairs \((u,v)\in V(G)\times V(H)\), and two such pairs
are adjacent if either their first coordinates are equal and their second
coordinates are adjacent in \(H\), or their second coordinates are equal and
their first coordinates are adjacent in \(G\). In particular,
\[
Q_{d+1}=Q_d\square K_2,
\]
where \(K_2\) is the graph on two vertices joined by an edge.

\begin{proposition}[Dimension lift]
	\label{prop:lift}
	For every \(d\ge1\) and every \(k\ge1\),
	\[
	\sigma_k(Q_{d+1})\ge \sigma_k(Q_d).
	\]
\end{proposition}

\begin{proof}
	Let \(I\) be the attacked parity class of \(Q_d\), and let
	\[
	c:Q_d\longrightarrow[q]
	\]
	be an \((I,k,q)\)-defensive silver colouring. We construct a colouring of
	\(Q_{d+1}=Q_d\square K_2\).
	
	Fix a coordinate vector \(e_1\in\mathbb F_2^d\), and define the folding map
	\[
	\phi:V(Q_{d+1})\longrightarrow V(Q_d),
	\qquad
	\phi(x,t)=x+t e_1,
	\qquad t\in\{0,1\}.
	\]
	Define the lifted colouring \(c'\) by
	\[
	c'(x,t)=c(\phi(x,t)).
	\]
	Let \(\varepsilon\in\{0,1\}\) be such that
	\[
	I=\{x\in\mathbb F_2^d:\wt(x)\equiv\varepsilon\pmod2\}.
	\]
	The corresponding parity class in \(Q_{d+1}=Q_d\square K_2\) is
	\[
	I'
	=
	\{(x,t)\in\mathbb F_2^d\times\mathbb F_2:
	\wt(x)+t\equiv\varepsilon\pmod2\}.
	\]
	Since
	\[
	\wt(\phi(x,t))
	=\wt(x+t e_1)
	\equiv\wt(x)+t\pmod2,
	\]
	we have
	\[
	I'=\phi^{-1}(I).
	\]
	Thus \(I'\) is precisely the parity class in \(Q_{d+1}\) corresponding to
	\(I\), and \(\phi\) maps \(I'\) two-to-one onto \(I\).
	The two vertices in the fibre above \(p\in I\) are
	\[
	(p,0)
	\qquad\text{and}\qquad
	(p+e_1,1).
	\]
	They have the same parity and may therefore occur together in an attack.
	By contrast, \((x,0)\) and \((x,1)\) have opposite parity, so they cannot
	belong to the same attack.
	
	The map \(\phi\) sends both types of edges of
	\(Q_d\square K_2\) to an edge of \(Q_d\). Therefore, if \(c\) is proper, then
	\(c'\) is also proper.
	
	Now let \(A\subseteq I'\) be an attack with \(|A|\le k\), and put
	\[
	P=\phi(A)\subseteq I.
	\]
	Since \(|P|\le|A|\le k\), the colouring \(c\) provides, for every colour
	\(i\in[q]\), an injective defence
	\[
	p\longmapsto z_p\in C_i
	\qquad (p\in P),
	\]
	where \(z_p\in N[p]\).
	
	For a vertex \((x,t)\in A\), put \(p=\phi(x,t)=x+t e_1\) and define the
	corresponding lift of \(z_p\) by
	\[
	\widetilde z_p^{\,t}=(z_p+t e_1,t).
	\]
	This vertex has colour \(i\), because
	\[
	c'(\widetilde z_p^{\,t})
	=c\bigl((z_p+t e_1)+t e_1\bigr)
	=c(z_p)=i.
	\]
	Moreover, since \(z_p\in N[p]\), the vertex
	\(\widetilde z_p^{\,t}\) belongs to \(N[(x,t)]\): its distance from
	\((x,t)\) equals the distance from \(z_p\) to \(p\). If both vertices in the
	fibre above \(p\) belong to \(A\), we use the two distinct lifts
	\(\widetilde z_p^{\,0}\) and \(\widetilde z_p^{\,1}\). If only one belongs
	to \(A\), we use the corresponding lift.
	
	The projected defenders \(z_p\) are distinct for distinct \(p\), and therefore
	their lifts are also distinct. Hence every attack \(A\subseteq I'\) with
	\(|A|\le k\) has an injective defence in every colour class. Thus \(c'\) is an
	\((I',k,q)\)-defensive silver colouring of \(Q_{d+1}\). Consequently,
	\[
	\sigma_k(Q_{d+1})\ge \sigma_k(Q_d).
	\]
\end{proof}

\begin{theorem}[Four dyadic families]
	\label{thm:dyadic}
	If \(q=2^h\ge2\), then
	\[
	\begin{aligned}
		\sigma_2(Q_{2q-3})&=\sigma_2(Q_{2q-2})=q,\\
		\sigma_2(Q_{2q-1})&=\sigma_2(Q_{2q})=q+1.
	\end{aligned}
	\]
\end{theorem}

\begin{proof}
	First consider dimension \(2q-3\). Since \(q=2^h\), the binary-linear
	construction of \cref{thm:linear-criterion} gives a pair-defensive colouring
	of \(Q_{2q-3}\) with \(q\) colours. By the dimension-lift proposition
	\cref{prop:lift}, the same number of colours is attainable in \(Q_{2q-2}\).
	
	Now consider dimension \(2q-1\). Label the coordinates by the multiset
	consisting of one zero vector and two copies of every nonzero vector of
	\(\mathbb F_2^h\). Colour the attacked parity class with a new colour
	\(\infty\). For a vertex \(x\) in the defender parity class, define its
	colour by
	\[
	c(x)=\sum_j x_j\lambda_j.
	\]
	As in \cref{ex:q7}, every attacked vertex sees the colour \(\infty\) through
	self-defense, one finite colour once, and every other finite colour twice.
	The unique finite defender, when it occurs, is obtained by flipping the
	coordinate labelled by the zero vector. Since this coordinate is fixed, the
	resulting map is injective. Hence we obtain a pair-defensive colouring of
	\(Q_{2q-1}\) with \(q+1\) colours.
	
	Applying the dimension lift once more gives \(q+1\) colours in \(Q_{2q}\).
	In all four dimensions, these constructions attain the upper bound from
	\cref{thm:pair-bound}. Therefore the stated equalities hold.
\end{proof}

Thus the upper bound is sharp in infinitely many dimensions. However,
\cref{thm:dyadic} does not characterize all dimensions in which equality may
hold.

\section{Equality in the odd critical dimension}
\label{sec:odd-equality}

We now study the structure of colourings that attain the pair bound in the
odd critical dimension
\[
d=2q-3.
\]
Throughout this section, let
\[
A_i=C_i\cap I,
\qquad
B_i=C_i\cap D
\]
be the attacked and defender parts of colour \(i\), respectively. For
\(x\in I\), define
\[
n_i(x)=|N(x)\cap B_i|.
\]
Thus \(n_i(x)\) counts how many neighbours of \(x\) in the defender parity
have colour \(i\).

The vector
\[
(n_i(x))_{i=1}^q
\]
is called the \emph{open profile} at \(x\). The corresponding
\emph{closed profile} is
\[
(|N[x]\cap C_i|)_{i=1}^q.
\]
When we write a profile, we regard its entries up to permutation.

Let \(H_d\) denote the halved cube on \(D\): its vertex set is \(D\), and two
vertices are adjacent when their Hamming distance in \(Q_d\) is two. An
\emph{ordered near-partition} of \(D\) is a labelled list
\[
B_1,\ldots,B_q
\]
of pairwise disjoint subsets of \(D\) whose union is \(D\). Empty blocks are
allowed.

\begin{lemma}[Equality profile]
	\label{lem:equality-profile}
	In an extremal \(q\)-colouring of \(Q_{2q-3}\), every closed profile
	\[
	\bigl(|N[x]\cap C_i|\bigr)_{i=1}^q,
	\qquad x\in I,
	\]
	is a permutation of
	\[
	(1,1,2,\ldots,2).
	\]
	One of the two entries equal to \(1\) corresponds to the colour \(c(x)\)
	and is realized by \(x\) itself. The other corresponds to a unique
	defender in \(D\). These latter defenders define a perfect matching
	between \(I\) and \(D\).
\end{lemma}

\begin{proof}
    Equality holds in every step of the counting argument following
    \cref{thm:pair-bound}. Hence every nonsingleton entry in every closed profile
    is equal to two, and
    \[
    S=2m,
    \qquad
    s_i=|C_i|
    \quad\text{for every }i\in[q].
    \]
    
    Let \(r_x\) be the number of singleton entries in the closed profile at
    \(x\in I\). Since \(d=2q-3\), we have \(|N[x]|=d+1=2q-2\). The entries of
    the closed profile are positive, the singleton entries contribute \(1\), and
    all other entries contribute \(2\). Therefore
    \[
    2q-2
    =
    d+1
    =
    r_x+2(q-r_x),
    \]
    so \(r_x=2\).
    
    One of these singleton entries is always the colour of \(x\), represented by
    \(x\) itself. Thus each \(x\in I\) has exactly one further singleton defender,
    and this defender lies in \(D\). Pair Hall implies that the same vertex of
    \(D\) cannot be this unique singleton defender for two different attacked
    vertices. Hence these singleton defenders define an injection
    \[
    I\longrightarrow D.
    \]
    Since equality gives exactly \(m=|I|\) such incidences and \(|D|=m\), this
    injection is a bijection.
\end{proof}

\begin{theorem}[Odd extremal partition theorem]
    \label{thm:odd-partition}
    Extremal pair-defensive \(q\)-colourings of \(Q_{2q-3}\), with attacked
    parity \(I\), are in bijection with ordered near-partitions
    \[
    B_1,\ldots,B_q
    \]
    of \(D\) such that each nonempty induced graph \(H_d[B_i]\) is triangle-free
    and \((q-2)\)-regular.
    
    The extension from \(D\) to \(I\) is unique: each vertex \(x\in I\) receives
    the unique colour \(i\) for which
    \[
    N(x)\cap B_i=\varnothing .
    \]
\end{theorem}

\begin{proof}
    We first show that every extremal colouring gives such a partition of \(D\).
    Then we prove that any partition with the stated properties uniquely extends
    to an extremal colouring.
    
    Suppose first that we have an extremal colouring. By
    \cref{lem:equality-profile}, the open profile at every \(x\in I\) is a
    permutation of
    \[
    (0,1,2,\ldots,2).
    \]
    In particular, for every colour \(i\) and every \(x\in I\), we have
    \(n_i(x)\le2\).
    
    Three vertices in the same parity class of the cube form a triangle in the
    halved cube exactly when they have a common neighbour in the opposite parity
    class. Therefore, if \(H_d[B_i]\) contained a triangle, some attacked vertex
    \(x\in I\) would have at least three neighbours in \(B_i\), contradicting
    \(n_i(x)\le2\). Hence each \(H_d[B_i]\) is triangle-free.
    
    Now fix a vertex \(z\in B_i\). Among the \(d\) attacked neighbours of \(z\),
    the vertex \(z\) is the unique \(B_i\)-neighbour for exactly one of them,
    namely the attacked vertex matched to \(z\) in
    \cref{lem:equality-profile}. For every other attacked neighbour \(x\) of
    \(z\), there is exactly one further vertex of \(B_i\) adjacent to \(x\).
    
    Each edge of \(H_d[B_i]\) incident with \(z\) corresponds to a vertex of
    \(B_i\) at Hamming distance two from \(z\). Such a pair has exactly two
    common neighbours in \(I\). Hence
    \[
    2\deg_{H_d[B_i]}(z)
    =
    \sum_{x\in N(z)}\bigl(n_i(x)-1\bigr).
    \]
    The sum has one zero term and \(d-1\) terms equal to one, so
    \[
    2\deg_{H_d[B_i]}(z)=d-1=2q-4.
    \]
    Therefore
    \[
    \deg_{H_d[B_i]}(z)=q-2.
    \]
    Thus every nonempty block \(B_i\) induces a triangle-free
    \((q-2)\)-regular subgraph of \(H_d\).
    
    Conversely, suppose that \(B_1,\ldots,B_q\) is an ordered near-partition of
    \(D\) such that every nonempty \(H_d[B_i]\) is triangle-free and
    \((q-2)\)-regular. Triangle-freeness implies
    \[
    n_i(x)\le2
    \]
    for every \(x\in I\) and every \(i\in[q]\), since three neighbours of
    \(x\) in the same block would form a triangle in the halved cube.
    
    Fix \(z\in B_i\). As above,
    \[
    2\deg_{H_d[B_i]}(z)
    =
    \sum_{x\in N(z)}\bigl(n_i(x)-1\bigr).
    \]
    Since \(H_d[B_i]\) is \((q-2)\)-regular and \(d=2q-3\), the left-hand side is
    \[
    2(q-2)=2q-4=d-1.
    \]
    There are \(d\) terms in the sum, and each term is either \(0\) or \(1\),
    because \(1\le n_i(x)\le2\) whenever \(x\in N(z)\). Hence exactly one of
    these terms is zero. Equivalently, each \(z\in B_i\) is the unique
    \(B_i\)-neighbour of exactly one attacked vertex.
    
    Summing over all blocks, there are exactly \(|D|=2^{d-1}\)
    singleton pairs \((x,i)\) with \(n_i(x)=1\). For a fixed \(x\in I\), let
    \(z_x\) and \(o_x\) denote the numbers of zero and one entries, respectively,
    in the open profile
    \[
    (n_i(x))_{i=1}^q.
    \]
    Since the blocks partition \(D\),
    \[
    \sum_{i=1}^q n_i(x)=d=2q-3.
    \]
    All entries are at most \(2\), so
    \[
    2z_x+o_x=3.
    \]
    In particular, \(o_x\ge1\) for every \(x\in I\). On the other hand, the total
    number of singleton pairs is \(|I|\), and therefore
    \[
    \sum_{x\in I}o_x=|I|.
    \]
    It follows that \(o_x=1\) for every \(x\in I\), and hence \(z_x=1\) for every
    \(x\in I\).
    
    We now extend the colouring to \(I\). Colour each \(x\in I\) with the unique
    colour \(i\) satisfying
    \[
    N(x)\cap B_i=\varnothing.
    \]
    This is the unique zero entry in the open profile at \(x\).
    
    The resulting colouring is proper. Indeed, if \(x\) receives colour \(i\),
    then \(x\) has no neighbour in \(B_i\), and therefore no neighbour of its own
    colour. Every colour class is nonempty. If \(B_i\neq\varnothing\), then colour \(i\)
    already appears in \(D\). If \(B_i=\varnothing\), then every vertex of \(I\)
    is assigned colour \(i\), so the colour is still used.
    
    Singleton Hall follows from the open profile: every colour either appears
    among the neighbours of \(x\), or is the colour assigned to \(x\) itself.
    Finally, suppose that pair Hall fails for some colour \(i\). Then two distinct
    attacked vertices would have the same unique colour-\(i\) defender. A defender
    in \(D\) cannot have this property by the singleton uniqueness established
    above. A defender in \(I\) can defend only itself, because \(I\) is
    independent. Thus pair Hall cannot fail.
    
    The extension is therefore a proper pair-defensive colouring. Since the colour
    assigned to each \(x\in I\) is forced by its unique zero entry, the two
    constructions are inverse.
\end{proof}

Empty blocks are allowed and can occur naturally. For example, when \(q=3\)
and \(d=3\), the defender parity can be divided into blocks of sizes
\[
2,\ 2,\ 0.
\]
Indeed, take \(D\) to be the even parity class of \(Q_3\), and define
\[
B_1=\{000,011\},\qquad
B_2=\{101,110\},\qquad
B_3=\varnothing.
\]
The halved cube on \(D\) is \(K_4\), so each nonempty block induces a
single edge and is therefore triangle-free and \(1\)-regular. Thus these
blocks satisfy the conditions of \cref{thm:odd-partition}, and the unique
extension colours every vertex of the attacked parity with colour \(3\).
This gives an extremal \(3\)-colouring of \(Q_3\) with defender-block sizes
\(2,2,0\).
The empty block simply corresponds to a colour that is absent from the
neighbourhood of every attacked vertex; that colour is then assigned to all
vertices of the attacked parity.

It is helpful to view each nonempty block \(B_i\) as the set of available
defenders of colour \(i\). Locally, such a block is a two-fold radius-one
packing: it meets every radius-one ball in the opposite parity class in at
most two vertices. For subsets of one parity class, this is equivalent to
saying that the induced subgraph of the halved cube is triangle-free
\cite{krotov2021multifold}. This describes the possible individual blocks,
but not how the blocks fit together.

The content of \cref{thm:odd-partition} is the global organization of these
blocks. The blocks \(B_1,\ldots,B_q\), including possible empty blocks, must
form an ordered near-partition of the entire defender parity. Once this
labelled partition is given, every attacked vertex sees exactly one missing
colour, and that missing colour determines its colour uniquely. Thus the
partition on the defender side contains enough information to reconstruct
the full pair-defensive colouring.

This viewpoint is different from the code partitions studied by Krotov and
Valyuzhenich \cite[Section~4.3.1 and Theorem~6]{krotov2024degree3}. They
classify partitions of the full cube \(Q_7\) into multifold
\(1\)-perfect codes. A \(\mu\)-fold \(1\)-perfect code meets every closed
radius-one ball in exactly \(\mu\) vertices, so its coverage number is the
same throughout the cube.

In our setting, the blocks cover only one parity class, namely the defender
parity. Their numbers of representatives in neighbourhoods of attacked
vertices may vary, and the missing colour at each attacked vertex is part of
the reconstruction process. Thus both settings involve partitions into
code-like objects and study neighbourhood incidences, but our blocks have a
different coverage requirement and serve as the defender-side data of a
proper pair-defensive colouring.

The full ordered near-partition also determines the sizes of the attacked
colour classes. The next identity uses this global colouring structure; it is
not a statement about an arbitrary two-fold radius-one packing.

Recall that
\[
A_i=C_i\cap I
\qquad\text{and}\qquad
B_i=C_i\cap D
\]
are, respectively, the attacked and defender parts of colour \(i\).

\begin{proposition}[Colour-class identity and empty blocks]
	\label{prop:odd-sizes}
	For every colour \(i\) in an extremal colouring of \(Q_{2q-3}\),
	\begin{equation}
		\label{eq:odd-size-identity}
		|A_i|
		=
		2^{d-1}-(q-1)|B_i|.
	\end{equation}
	If some defender block is empty, then \(q=2^u+1\) for some \(u\ge1\).
	Moreover, exactly one defender block is empty, every nonempty defender
	block has size
	\[
	\frac{2^{d-1}}{q-1},
	\]
	and all attacked colour classes except the one corresponding to the empty
	defender block are empty.
\end{proposition}

\begin{proof}
	Fix a colour \(i\). In the odd equality case, the open profile at every
	\(x\in I\) is a permutation of
	\[
	(0,1,2,\ldots,2).
	\]
	For colour \(i\), the value \(n_i(x)\) is zero precisely when \(x\in A_i\).
	It is one precisely when \(x\) is the unique attacked vertex matched to a
	vertex of \(B_i\). By \cref{lem:equality-profile}, this occurs for exactly
	\(|B_i|\) vertices of \(I\). At every remaining vertex of \(I\), the value of
	\(n_i(x)\) is two. Therefore
	\[
	\begin{aligned}
		\sum_{x\in I} n_i(x)
		&=
		0\cdot |A_i|
		+1\cdot |B_i|
		+2\bigl(2^{d-1}-|A_i|-|B_i|\bigr)\\
		&=
		2^d-2|A_i|-|B_i|.
	\end{aligned}
	\]
	On the other hand, this sum counts the cube edges between \(I\) and \(B_i\).
	Since every vertex of \(B_i\) has degree \(d\), the same quantity equals
	\(d|B_i|\). Hence
	\[
	d|B_i|=2^d-2|A_i|-|B_i|.
	\]
	Using \(d+1=2(q-1)\), we obtain
	\[
	|A_i|=2^{d-1}-(q-1)|B_i|,
	\]
	which proves \eqref{eq:odd-size-identity}.
	
	Now suppose that \(B_i=\varnothing\). Then colour \(i\) is absent from the
	open neighbourhood of every vertex in \(I\). Since each attacked vertex
	receives the unique colour missing from its open neighbourhood, every vertex
	of \(I\) has colour \(i\). Thus
	\[
	A_i=I.
	\]
	Consequently, \(A_j=\varnothing\) for every \(j\ne i\). Applying
	\eqref{eq:odd-size-identity} to such a colour \(j\) gives
	\[
	|B_j|=\frac{2^{d-1}}{q-1}.
	\]
	Therefore \(q-1\) divides \(2^{d-1}\), so \(q-1\) is a power of two. Hence
	\(q=2^u+1\) for some \(u\ge1\).
	
	Finally, two defender blocks cannot both be empty, since each empty block
	would force all vertices of \(I\) to receive its colour.
\end{proof}

Only the forward implication is claimed here: the condition \(q=2^u+1\) does
not force an empty defender block. For example, a silver lift with
\(q=5\) and \(d=7\) can have defender-block sizes
\[
8,\ 8,\ 16,\ 16,\ 16.
\]

There is also a useful component version of the partition theorem. Split each
defender block into its connected components in the halved cube. Define the
\emph{conflict graph} as follows: its vertices are these components, and two
components are adjacent if there is a halved-cube edge between them.

\begin{corollary}[Conflict-graph formulation]
    \label{cor:conflict-odd}
    We have
    \[
    \sigma_2(Q_{2q-3})=q
    \]
    if and only if the defender halved cube \(H_d\) admits a partition into
    connected induced triangle-free \((q-2)\)-regular subgraphs whose conflict
    graph is \(q\)-colourable.
    
    Every such partition has at least \(q-1\) components.
\end{corollary}

\begin{proof}
    Suppose first that \(\sigma_2(Q_{2q-3})=q\). By
    \cref{thm:odd-partition}, the defender parity \(D\) has an ordered
    near-partition
    \[
    B_1,\ldots,B_q
    \]
    such that each nonempty \(H_d[B_i]\) is triangle-free and
    \((q-2)\)-regular. Split each nonempty block into its connected components.
    These components are connected induced triangle-free \((q-2)\)-regular
    subgraphs of \(H_d\). Moreover, two components coming from the same block
    cannot be adjacent in \(H_d\), since otherwise they would lie in the same
    connected component of \(H_d[B_i]\). Thus the original labels
    \(1,\ldots,q\) give a proper \(q\)-colouring of the conflict graph.
    
    Conversely, suppose \(H_d\) is partitioned into connected induced
    triangle-free \((q-2)\)-regular subgraphs and that the resulting conflict
    graph has a proper \(q\)-colouring. Group the components according to their
    conflict-graph colours. Since adjacent components receive different colours,
    each group induces a disjoint union of triangle-free \((q-2)\)-regular
    components, and hence is itself triangle-free and \((q-2)\)-regular. These
    groups form the ordered near-partition required by
    \cref{thm:odd-partition}, so they determine an extremal pair-defensive
    \(q\)-colouring.
    
    Finally, \eqref{eq:odd-size-identity} gives
    \[
    |B_i|\le \frac{2^{d-1}}{q-1}
    =
    \frac{2^d}{d+1}
    \]
    for every block \(B_i\). Therefore each connected component has at most this
    many vertices. Since the components cover all \(2^{d-1}\) vertices of the
    defender parity, their number is at least
    \[
    \frac{2^{d-1}}{2^{d-1}/(q-1)}
    =
    q-1.
    \]
\end{proof}

The same numerical bound can also be obtained from Hoffman's bound. Indeed,
the least eigenvalue of the odd halved cube is \(-(q-2)\), which is the
negative of the induced degree of each block. The spectral ratio bound
therefore gives exactly
\[
|B|\le \frac{2^d}{d+1}.
\]
This is equivalent to the nonnegativity condition
\[
|A_i|\ge0
\]
in \eqref{eq:odd-size-identity}, and so it gives no additional obstruction.

\section{Local matching and defect geometry}
\label{sec:local-geometry}

The partition theorem can also be expressed in coordinates. Let \(B\) be a
nonempty defender block. For \(b\in B\), define the \emph{link} of \(b\) in
\(B\) to be the graph with vertex set \([d]\) and edge set
\begin{equation}
    \label{eq:link}
    L_b
    =
    \bigl\{
    \{j,\ell\}\subseteq[d]:
    j\ne \ell
    \text{ and }
    b+e_j+e_\ell\in B
    \bigr\}.
\end{equation}
Thus an edge \(\{j,\ell\}\) in \(L_b\) records that \(B\) contains the
halved-cube neighbour of \(b\) obtained by flipping coordinates \(j\) and
\(\ell\).

A coordinate pair always means a two-element subset of \([d]\). A
\emph{near-perfect matching} of \([d]\) is a matching that covers all but one
coordinate. The next lemma shows that every link is such a matching. We call
the unique uncovered coordinate the \emph{defect coordinate}.

\begin{lemma}[Local defect]
    \label{lem:local-defect}
    Let \(B\) induce a triangle-free \((q-2)\)-regular subgraph of the halved
    \((2q-3)\)-cube. For every \(b\in B\), the link \(L_b\) is a near-perfect
    matching of \([d]\). Equivalently, there is a unique coordinate
    \(\delta(b)\in[d]\) such that \(L_b\) is a perfect matching of
    \[
    [d]\setminus\{\delta(b)\}.
    \]
    Moreover, for every \(j\in[d]\),
    \begin{equation}
        \label{eq:local-count}
        |N(b+e_j)\cap B|=1+\deg_{L_b}(j).
    \end{equation}
    Thus \(b+e_{\delta(b)}\) is the unique attacked neighbour of \(b\) whose only
    \(B\)-defender is \(b\).
\end{lemma}

\begin{proof}
    The edges of \(L_b\) correspond exactly to the neighbours of \(b\) inside the
    induced halved-cube subgraph \(H_d[B]\). Since this subgraph is
    \((q-2)\)-regular, the link \(L_b\) has \(q-2\) edges.
    
    We claim that \(L_b\) is a matching. If two edges of \(L_b\) shared a
    coordinate, say
    \[
    \{j,\ell\}
    \quad\text{and}\quad
    \{j,m\},
    \]
    then \(B\) would contain both
    \[
    b+e_j+e_\ell
    \quad\text{and}\quad
    b+e_j+e_m.
    \]
    These two vertices differ in the two coordinates \(\ell\) and \(m\), so they
    are adjacent in the halved cube. Together with \(b\), they form a triangle in
    \(H_d[B]\), contradicting triangle-freeness.
    
    Thus \(L_b\) is a matching with \(q-2\) edges. It covers
    \[
    2(q-2)=2q-4=d-1
    \]
    coordinates, so exactly one coordinate is left uncovered. This is the defect
    coordinate \(\delta(b)\).
    
    Finally, fix \(j\in[d]\). The vertices of \(B\) adjacent in \(Q_d\) to
    \(b+e_j\) are precisely \(b\) itself and the vertices of the form
    \[
    b+e_j+e_\ell
    \]
    for which \(\{j,\ell\}\) is an edge of \(L_b\). Hence
    \[
    |N(b+e_j)\cap B|=1+\deg_{L_b}(j),
    \]
    as claimed.
\end{proof}

The coordinate-matching idea has clear precedents. The proof of Lemma~5.2 of
\emph{Silver Cubes} derives matchings, together with a unique unsaturated
coordinate, from a proper edge-colouring of a coordinate complete graph
\cite{ghebleh2008silver}. Similarly, the proof of Theorem~7 of
Krotov--Potapov gives a pointwise perfect matching for even unitrades
\cite{krotov2021multifold}. In the present setting, the new feature is the
role of the odd defect: it identifies the unique attack for which \(b\) is
the sole defender from its block, and this local fact has the global
consequences developed below.

\begin{proposition}[Singleton matching and coherence]
    \label{prop:coherence}
    For an extremal colouring of \(Q_{2q-3}\), the following hold.
    
    \begin{enumerate}[label=\textup{(\alph*)}]
        \item For \(b\in D\), define
        \[
        \psi(b)=b+e_{\delta(b)}.
        \]
        Then \(\psi:D\to I\) is a bijection. Thus the defect edges
        \[
        b\psi(b),
        \qquad b\in D,
        \]
        form a perfect matching between the two parity classes.
        
        \item For \(x\in I\), define
        \[
        \delta(x)=\delta(\psi^{-1}(x)).
        \]
        Also define
        \[
        \Lambda_x
        =
        \bigl\{
        \{j,\ell\}\subseteq[d]:
        j\ne \ell
        \text{ and }
        c(x+e_j)=c(x+e_\ell)
        \bigr\}.
        \]
        Then \(\Lambda_x\) is a perfect matching of
        \[
        [d]\setminus\{\delta(x)\}.
        \]
        
        \item Adjoin one auxiliary coordinate \(*\). For each vertex \(v\in V(Q_d)\),
        define
        \[
        \Pi_v^+
        =
        \begin{cases}
            L_v\cup\{\{\delta(v),*\}\}, & v\in D,\\[2mm]
            \Lambda_v\cup\{\{\delta(v),*\}\}, & v\in I.
        \end{cases}
        \]
        Then \(\Pi_v^+\) is a perfect matching of the enlarged coordinate set
        \[
        [d]\cup\{*\}.
        \]
        Let \(\pi_v(j)\) denote the mate of \(j\) in this matching. Then, for every
        \(v\in\mathbb F_2^d\) and every \(j\in[d]\),
        \begin{equation}
            \label{eq:coherence}
            \pi_{v+e_j}(j)=\pi_v(j).
        \end{equation}
    \end{enumerate}
\end{proposition}

\begin{proof}
    By \cref{lem:local-defect}, every defender \(b\in D\) is the unique
    same-colour defender of exactly one attacked vertex, namely
    \[
    b+e_{\delta(b)}.
    \]
    By \cref{lem:equality-profile}, every attacked vertex has exactly one
    singleton defender in \(D\). These two facts prove that
    \[
    \psi(b)=b+e_{\delta(b)}
    \]
    is a bijection from \(D\) to \(I\), proving \textup{(a)}.
    
    Now fix \(x\in I\). By the equality profile, the \(d=2q-3\) neighbours of
    \(x\) have one colour appearing once and \(q-2\) colours appearing twice.
    The singleton neighbour is \(\psi^{-1}(x)\), and its coordinate is
    \(\delta(x)\). Therefore, after removing the coordinate \(\delta(x)\), the
    remaining \(d-1=2(q-2)\) coordinates split into pairs according to which
    neighbours of \(x\) have the same colour. Hence \(\Lambda_x\) is a perfect
    matching of
    \[
    [d]\setminus\{\delta(x)\},
    \]
    proving \textup{(b)}.
    
    It remains to prove the coherence condition. Let \(v\in D\), write
    \(v=b\), and let
    \[
    x=b+e_j\in I.
    \]
    Suppose first that \(j\) is paired with \(\ell\) in \(L_b\). Then
    \[
    b+e_j+e_\ell\in B_i,
    \]
    where \(b\in B_i\). Thus the two vertices
    \[
    b
    \quad\text{and}\quad
    b+e_j+e_\ell
    \]
    are the two \(B_i\)-neighbours of \(x\). Hence the two neighbours
    \(x+e_j=b\) and \(x+e_\ell=b+e_j+e_\ell\) have the same colour, so \(j\) is
    paired with \(\ell\) in \(\Lambda_x\).
    
    If instead \(j=\delta(b)\), then \(b\) is the unique \(B_i\)-neighbour of
    \(x\). Thus \(j\) is paired with the auxiliary coordinate \(*\) at both
    endpoints:
    \[
    \pi_b(j)=*
    =
    \pi_x(j).
    \]
    
    The same arguments, read in reverse, show that if the partner of \(j\) is
    specified at \(x\), then it is the same at \(b\). Therefore the two endpoints
    of the edge in direction \(j\) agree on the mate of \(j\), that is,
    \[
    \pi_{v+e_j}(j)=\pi_v(j)
    \]
    for every \(v\in\mathbb F_2^d\) and every \(j\in[d]\). This proves
    \textup{(c)}.
\end{proof}

Equation~\eqref{eq:coherence} says that the two endpoints of a cube edge
agree on the partner of the edge direction. In particular, for a pair
\[
P=\{j,\ell\}\subseteq [d]\cup\{*\},
\]
the set
\[
\{v:P\in\Pi_v^+\}
\]
is invariant under translation by \(e_j\) and by \(e_\ell\), when both
\(j,\ell\in[d]\). If \(\ell=*\), it is invariant under translation by
\(e_j\). No further consistency condition around closed walks is required,
because the perfect matching \(\Pi_v^+\) is already defined independently at
each vertex \(v\).

\begin{lemma}[Defect congruences]
    \label{lem:defect-congruence}
    Let \(B\) induce a triangle-free \((q-2)\)-regular subgraph of the halved
    \((2q-3)\)-cube. For a coordinate \(j\in[d]\), put
    \[
    B_j^e=\{v\in B:v_j=e\},
    \qquad e\in\{0,1\},
    \]
    and define
    \[
    n_j^e
    =
    |\{v\in B_j^e:\delta(v)=j\}|,
    \qquad
    n_j=n_j^0+n_j^1.
    \]
    Then
    \[
    n_j\equiv |B|\pmod 2.
    \]
    If \(q\) is even, then
    \[
    n_j\equiv |B|\pmod 4,
    \qquad
    n_j^e\equiv |B_j^e|\pmod 2
    \quad(e\in\{0,1\}).
    \]
    If \(q\) is odd, then both \(n_j^0\) and \(n_j^1\) are even.
\end{lemma}

\begin{proof}
    Fix a coordinate \(j\). For each \(\ell\ne j\), translation by
    \(e_j+e_\ell\) maps a vertex \(v\in B\) with
    \(\{j,\ell\}\in L_v\) to the corresponding halved-cube neighbour
    \[
    v+e_j+e_\ell\in B.
    \]
    This translation is a fixed-point-free involution on the set
    \[
    \{v\in B:\{j,\ell\}\in L_v\}.
    \]
    Hence this set has even size. Summing over all \(\ell\ne j\), we find that
    the number of vertices \(v\in B\) for which \(j\) is covered by \(L_v\) is
    even. Since \(j\) is uncovered exactly when \(\delta(v)=j\), this gives
    \[
    |B|-n_j\equiv0\pmod 2,
    \]
    and therefore
    \[
    n_j\equiv |B|\pmod 2.
    \]
    
    For the refinement, let \(m_j\) be the number of edges of \(H_d[B]\) whose
    coordinate pair contains \(j\). These are exactly the edges crossing between
    the two fibres
    \[
    B_j^0
    \quad\text{and}\quad
    B_j^1.
    \]
    A vertex \(v\in B_j^e\) is incident with such a crossing edge exactly when
    \(\delta(v)\ne j\). Hence
    \[
    m_j=|B_j^0|-n_j^0=|B_j^1|-n_j^1,
    \qquad
    |B|-n_j=2m_j.
    \]
    
    Now count degrees inside one fibre \(B_j^e\). Each vertex with
    \(\delta(v)=j\) has all its \(q-2\) halved-cube neighbours inside the same
    fibre. Each vertex with \(\delta(v)\ne j\) has exactly one neighbour across
    the \(j\)-cut, and therefore \(q-3\) neighbours inside the fibre. Thus the
    degree sum inside \(B_j^e\) is
    \[
    (q-3)\bigl(|B_j^e|-n_j^e\bigr)+(q-2)n_j^e.
    \]
    This degree sum is even, so
    \[
    (q-3)\bigl(|B_j^e|-n_j^e\bigr)+(q-2)n_j^e
    \equiv 0 \pmod 2.
    \]
    Using \(m_j=|B_j^e|-n_j^e\), this is equivalently
    \[
    (q-3)m_j+(q-2)n_j^e\equiv0\pmod2.
    \]
    
    If \(q\) is even, then \(q-3\) is odd and \(q-2\) is even, so the congruence
    implies that \(m_j\) is even. From
    \[
    |B|-n_j=2m_j
    \]
    we get
    \[
    n_j\equiv |B|\pmod4.
    \]
    Also, from
    \[
    m_j=|B_j^e|-n_j^e
    \]
    and the fact that \(m_j\) is even, we obtain
    \[
    n_j^e\equiv |B_j^e|\pmod2.
    \]
    
    If \(q\) is odd, then \(q-3\) is even and \(q-2\) is odd. The same congruence
    therefore gives
    \[
    n_j^e\equiv0\pmod2
    \]
    for each \(e\in\{0,1\}\).
\end{proof}

We can also express equality in the odd critical dimension by a single local
condition.

\begin{corollary}[Local normal form]
    \label{cor:local-normal-form}
    Extremal pair-defensive \(q\)-colourings of \(Q_{2q-3}\) are in bijection
    with maps
    \[
    c_D:D\to[q]
    \]
    such that, for every \(b\in D\), the set
    \[
    \bigl\{
    \{j,\ell\}\subseteq[d]:
    j\ne \ell
    \text{ and }
    c_D(b+e_j+e_\ell)=c_D(b)
    \bigr\}
    \]
    is a near-perfect matching of \([d]\).
    
    For such a map \(c_D\), the colours on \(I\), the defect matching between
    \(D\) and \(I\), and all pair-Hall inequalities are forced.
\end{corollary}

\begin{proof}
    The forward implication is exactly \cref{lem:local-defect}, applied to each
    defender block.
    
    Conversely, suppose \(c_D:D\to[q]\) satisfies the local matching condition.
    For each colour \(i\), let
    \[
    B_i=c_D^{-1}(i).
    \]
    These sets form an ordered near-partition of \(D\), with empty blocks
    allowed. The local matching condition says that every vertex in a nonempty
    block \(B_i\) has exactly \(q-2\) neighbours in \(H_d[B_i]\), because a
    near-perfect matching of \([d]\) has
    \[
    \frac{d-1}{2}=q-2
    \]
    edges.
    
    It also implies that \(H_d[B_i]\) is triangle-free. Indeed, if two
    halved-cube neighbours of \(b\in B_i\) inside \(B_i\) were adjacent to each
    other, then the corresponding two coordinate pairs would share one
    coordinate. This would contradict the fact that the link at \(b\) is a
    matching.
    
    Thus every nonempty \(H_d[B_i]\) is triangle-free and
    \((q-2)\)-regular. By \cref{thm:odd-partition}, the ordered near-partition
    extends uniquely to an extremal pair-defensive colouring of \(Q_{2q-3}\).
\end{proof}

We call an extremal colouring \emph{homogeneous} if the link \(L_b\) is the
same for every \(b\in D\). The family
\[
\{L_b:b\in D\}
\]
is called the \emph{local matching field}. In particular, homogeneity forces
the defect coordinate to be constant. The next normal form describes all
constant-defect colourings, without assuming that the whole local matching
field is constant.

\begin{proposition}[Constant-defect reduction]
    \label{prop:constant-defect}
    Suppose that, in an extremal colouring of \(Q_{2q-3}\), the defect coordinate
    is constant:
    \[
    \delta(b)=j_0
    \qquad\text{for every } b\in D.
    \]
    Deleting coordinate \(j_0\) identifies the parity class \(D\) with
    \(\mathbb F_2^{2q-4}\). Under this identification, such colourings correspond
    to maps
    \[
    \gamma:\mathbb F_2^{2q-4}\to[q]
    \]
    satisfying the following two conditions for every \(v\in\mathbb F_2^{2q-4}\):
    
    \begin{enumerate}[label=\textup{(\roman*)}]
        \item \(\gamma(v+e_j)\ne\gamma(v)\) for every remaining coordinate \(j\);
        
        \item the set
        \[
        \bigl\{
        \{j,\ell\}:
        j\ne \ell
        \text{ and }
        \gamma(v+e_j+e_\ell)=\gamma(v)
        \bigr\}
        \]
        is a perfect matching of the \(2q-4\) remaining coordinates.
    \end{enumerate}
\end{proposition}

\begin{proof}
    Assume first that the defect coordinate is constant and equal to \(j_0\).
    Since \(j_0\) is the uncovered coordinate in every link, no link edge contains
    \(j_0\). Therefore all halved-cube edges inside a defender block avoid the
    coordinate \(j_0\). In particular, each connected component of a block lies
    entirely inside one \(j_0\)-fibre.
    
    Delete the coordinate \(j_0\), and identify \(D\) with
    \(\mathbb F_2^{2q-4}\). A halved-cube edge whose coordinate pair avoids
    \(j_0\) remains a distance-two edge after deletion. A halved-cube edge whose
    coordinate pair contains \(j_0\) becomes a distance-one edge after deletion.
    
    Now let \(\gamma\) be the induced colouring on
    \(\mathbb F_2^{2q-4}\). If for some remaining coordinate \(j\) we had
    \[
    \gamma(v+e_j)=\gamma(v),
    \]
    then, before deleting \(j_0\), these two vertices would differ in the two
    coordinates \(j_0\) and \(j\). Hence they would be adjacent in the original
    halved cube and would lie in the same defender block. This would give a block
    edge using \(j_0\), contradicting the fact that \(j_0\) is the defect
    coordinate everywhere. Thus condition \textup{(i)} holds.
    
    Condition \textup{(ii)} is exactly the link condition after deleting
    \(j_0\): the \(q-2\) link pairs at each defender cover all
    \[
    2q-4
    \]
    remaining coordinates and form a perfect matching.
    
    Conversely, suppose that
    \[
    \gamma:\mathbb F_2^{2q-4}\to[q]
    \]
    satisfies \textup{(i)} and \textup{(ii)}. Reinsert the deleted coordinate
    \(j_0\) to recover the defender parity \(D\). Condition \textup{(i)} ensures
    that no same-colour halved-cube edge uses \(j_0\). Condition \textup{(ii)}
    then gives, at every defender, a near-perfect link whose unique uncovered
    coordinate is \(j_0\). By \cref{cor:local-normal-form}, this defender-side
    colouring extends uniquely to an extremal pair-defensive colouring of
    \(Q_{2q-3}\).
\end{proof}

\section{Conclusion}

Pair defense requires each colour class to supply distinct nearby defenders
for any two attacked vertices. On hypercubes this limits the number of
colours to at most $\lfloor(d+3)/2\rfloor$, and the constructions here attain
that bound in four consecutive dimensions around every power of two. In the
odd critical dimension $d=2q-3$, attaining the bound is equivalent to an
ordered near-partition of the defender parity into blocks inducing
triangle-free $(q-2)$-regular subgraphs of the halved cube. The local matching
form and the constant-defect reduction give two ways to study such partitions.

\begingroup
\small
\setlength{\parskip}{0pt}
\bibliographystyle{plain}
\bibliography{references}
\endgroup

\end{document}